\documentclass{amsart}

\usepackage{amssymb,amsmath,amsthm,latexsym}
\usepackage{epsfig}
\usepackage{psfrag}
\usepackage{amscd}

\newtheorem{Theorem}{\bf Theorem}

\newtheorem{proposition}[Theorem]{\bf Proposition}

\newtheorem{theorem}[Theorem]{\bf Theorem}

\newcommand{\be}{\begin{equation}}
\newcommand{\ee}{\end{equation}}

\def\hpic #1 #2 {\mbox{$\begin{array}[c]{l} 
\epsfig{file=#1,height=#2}\end{array}$}}
\def\wpic #1 #2 {\mbox{$\begin{array}[c]{l} 
\epsfig{file=#1,width=#2}\end{array}$}}

\begin{document}

\title[On the genesis of a determinantal identity]{On the genesis of a determinantal identity}

\author{Vijay Kodiyalam}
\address{The Institute of Mathematical Sciences, Chennai, India}
\email{vijay@imsc.res.in}



\begin{abstract} 
We show how to derive a $3 \times 3$ determinantal identity in $12$ indeterminates that gives an 
explicit version of a result of Mohan Kumar and Pavaman Murthy on completing unimodular rows.
\end{abstract}

\dedicatory{This paper is dedicated to V. S. Sunder on his 60th birthday.}
\maketitle

\section{Introduction}
 
Let $A$ be a commutative ring with identity.
A sequence $a_1,a_2,\cdots,a_n$ of elements
of $A$ is said to be a unimodular sequence
if the ideal $(a_1,\cdots,a_n)A = A.$ Given such a unimodular sequence, there is a
projective (in fact, stably free) $A$-module, say $P$,
associated with it
and defined by the split short exact sequence:
%
$$
\CD
0 @>>> A @>{\left[ \begin{matrix} a_1 \\ \vdots \\ a_n \end{matrix} \right]}>> A^n @>>> P
@>>> 0.
\endCD
$$
%
It is well known and easy to see that $P$
is free exactly when the column
matrix above can be
completed to an invertible matrix over $A$,
in which case the unimodular sequence is said to be completable.

It is an interesting problem to find conditions under which a unimodular sequence
is completable. The first such non-trivial
result was by Swan and Towber in \cite{SwnTwb1975} who showed that if $a,b,c$
is unimodular then $a^2,b,c$ is completable.
In particular, they showed that this
follows (easily) from the determinantal identity
$$
det
\left[ \begin{array}{ccc}
 a^2 & b+ar &  c-aq\\
  b& -r^2+bpr & p+qr+cpr \\
  c& -p+qr-bpq & -q^2-cpq 
\end{array}
\right] = (pa+qb+rc)^2,
$$
that holds for commuting indeterminates $a,b,c,p,q,r$.

This result of Swan and Towber had a beautiful 
generalisation due to Suslin in \cite{Ssl1977} who showed that if $a_0,a_1,\cdots,a_n$ is unimodular and $r_0,r_1,\cdots,r_n \in {\mathbb N}$, then $a_0^{r_0},a_1^{r_1},a_2^{r_2},\cdots,a_n^{r_n}$ is completable provided
$n!$ divides $r_0r_1\cdots r_n$. While
Suslin's proof is, in principle, constructive and yields a completion
of the unimodular row, I am not aware of a
determinantal identity that implies it.
Suslin's result was generalised conjecturally by Nori and a special case of this conjecture was proved by Mohan Kumar in \cite{Kmr1997}.

 The main result of
\cite{KmrMrt2010} is another such condition
that applies to a length $3$ unimodular sequence $a,b,c$ and states that it is completable
if $z^2+bz+ac=0$ has its roots in $A$. It is natural to ask whether this result is an easy consequence of some determinantal
identity and the main observation of this
paper is an affirmative answer to this question.

\section{The determinantal identity}

The following determinantal identity can
be checked directly by a computation or more
easily using a computer algebra system such
as {\em Macaulay 2}.

\begin{proposition}\label{det} For indeterminates
$g,h,j,k,s,t,u,v,w,x,y,z$, the following determinantal identity holds.
\begin{eqnarray*}
det
\left[
\begin{array}{ccc}
gj & g^2j^2x+h^2j^2z-hjv & g^2k^2x+h^2k^2z+hju \\
gk+hj& gjt-hkv & hku-gjs \\
hk& g^2j^2w+h^2j^2y+gkt & g^2k^2w+h^2k^2y-gks
\end{array}
\right] =\\ (g^2j^2s+g^2k^2t+h^2j^2u+h^2k^2v)(g^2j^2w+g^2k^2x+h^2j^2y+h^2k^2z).
\end{eqnarray*}
\end{proposition}

An immediate corollary of Proposition \ref{det} is the following theorem.

\begin{theorem}[Theorem 3 of \cite{KmrMrt2010}]\label{main} Let $A$ be a ring and
$a,b,c \in A$ be a unimodular sequence such that $z^2+bz+ac=0$ has its roots in $A$. Then,
$a,b,c$ is completable.
\end{theorem}

\begin{proof}
Let $\alpha,\beta \in A$ be the roots of 
$z^2+bz+ac=0,$ so that $\alpha+\beta=-b$ and
$\alpha\beta=ac.$
Consider now the polynomial ring ${\mathbb Z}[g,h,j,k]$ and its subring ${\mathbb Z}[gj,gk,hj,hk],$ which is easily seen to be isomorphic to ${\mathbb Z}[W,X,Y,Z]/(WZ-XY)$.
There is thus a ring homomorphism from ${\mathbb Z}[gj,gk,hj,hk]$ to $A$ taking
$gj$ to $a$, $gk$ to $-\alpha$,$hj$ to  $-\beta$ and 
$hk$ to $c$.

Now since $a,b,c$ is unimodular, so is
$a,\alpha,\beta,c$ and so also is $a^2,\alpha^2,\beta^2,c^2$. Thus choose
$s,t,u,v,w,x,y,z \in A$ such that $a^2s + \alpha^2t
+\beta^2u + c^2v = 1=a^2w + \alpha^2x
+\beta^2y + c^2z$.
Finally, the image
of the $3 \times 3$ matrix of Proposition \ref{det} under the natural homomorphism to $A$ is seen to have determinant $1$ and
first column $a,b,c$, thus yielding an explicit completion of the unimodular sequence.
\end{proof}

\section{Derivation of the identity}

While the proof of Theorem \ref{main} is
mathematically complete and gives a third proof of the main result of \cite{KmrMrt2010} (their paper has two proofs), it is unsatisfying without an explanation of the genesis of the determinantal identity used. In this section,
we remedy this defect by explaining how such 
an identity was arrived at using the computer algebra system {\em Macaulay 2}.
Thus some arguments in this section depend
on the output of calculations of {\em Macaulay 2} and cannot be regarded as proofs
in the mathematical sense.

We begin by recalling the second proof of
Theorem \ref{main} from \cite{KmrMrt2010}.
Let $A$ be the ring ${\mathbb Z}[a,b,c,d,p,q,r]/(ad-bc,pa+q(b+c)+rd-1)$ and
let $P$ be the stably-free $A$ module associated to the unimodular sequence
$a,b+c,d$. It suffices to see that $P$ is free.
Since $P$ is of
rank $2$, to show that $P$ is free, it
suffices to map it onto an 2-generated invertible
ideal, say $I$, of $A$. For then, the kernel
of this map is another invertible ideal, say
$J$, of $A$. Since $det(P) \cong A$ (by stable freeness), it follows that $IJ$ is principal and so $J \cong I^{-1}$.
Thus $P \cong I \oplus I^{-1}$ which is free
by a lemma of Swan.
They then show that $P$ maps onto the invertible ideal $I = (a,b)^2A$.
What we will do is to make each step and map
in their proof as explicit as possible.
We make our first appeal to {\em Macaulay 2}
to check that $A$ is an integral domain.

They first observe that the ideal $(a,b)A$ is  an invertible (equivalently projective) ideal, for, localised at a maximal
ideal not containing either $a$ or $b$,
it becomes the unit ideal, while localising
at a maximal ideal not containing either $c$
or $d$ makes it principal and generated by
$b$ or $a$ respectively. However by unimodularity, no maximal ideal contains
all of $a,b,c,d$.
Next they note that $(a,b)^2A =
(a^2,b^2)A$ since this is true locally.
Alternatively, we may note that multiplying
the relation $pa+q(b+c)+rd=1$ by $ab$ and
using that $ad=bc$ gives $ab=a^2(pb+qd)+b^2(qa+rc).$

Next, we see that the inverse of the invertible ideal $(a,b)A$ is $(a,c)A$.
For, their product is $(a^2,ab,ac,bc)A =
(a^2,ab,ac,ad)A = a(a,b,c,d)A = aA \cong A$.
Thus the inverse of $I = (a^2,b^2)A$ is
$J = (a^2,c^2)A$ and their product is
$a^2A$.

We know now, in principle, that $A^2 \cong I \oplus J$. What we seek is an explicit
exact sequence
$$
\CD
0 @>>> J @>\iota>> A^2 @>\pi>> I @>>> 0,
\endCD
$$
together with a retract, say $\rho$, of $\iota$.
The surjection $\pi : A^2 \rightarrow I$ is easy to
describe. 
With $e_1,e_2$ denoting the standard basis
vectors in $A^2$, set 
$\pi(e_1) = a^2$
and $\pi(e_2) = b^2$. Let $K = ker(\pi).$

Now invoke {\em Macaulay 2} to see that $K$ is the submodule of $A^2$ generated by
the columns of the matrix
$$
\left[
\begin{array}{rr}
-b^2 & -d^2\\
a^2 & c^2
\end{array}
\right].
$$
There is a well-defined (and clearly injective) map $J \rightarrow K$ defined by $x \mapsto ye_1+xe_2$ where $d^2x+c^2y=0$. 
It is easily checked that this map is an
isomorphism. Thus the map $\iota : J \rightarrow A^2$ is given by $\iota(x) = ye_1+xe_2$ with $y$ as above. Suppose that $\rho : A^2 \rightarrow J$ is a splitting of this map.
Let $\rho(e_1) = -(ta^2 + vc^2)$ and $\rho(e_2) = sa^2 + uc^2$.  A simple calculation now shows that a necessary and sufficient condition for $\rho \circ \iota$ to be $id_J$, is that
$sa^2+tb^2+uc^2+vd^2=1$. We know that
such $s,t,u,v$ certainly exist in $A$
so choose them. The map $\rho \oplus \pi : A^2 \rightarrow J \oplus I$ is then an
isomorphism.

The next step in the analysis is to construct an explicit isomorphism from
$J \oplus I$ to $P$.
%
%
%
Since $P$ is, by definition, the quotient
of $A^3$ by the cyclic submodule generated by $ae_1 +(b+c)e_2+de_3$ (where $e_1,e_2,e_3$ are the standard basis vectors
of $A^3$), the surjection $A^3 \rightarrow I$ defined by
${e_1} \mapsto b^2$, ${e_2} \mapsto -ab$ and ${e_3} \mapsto a^2$ factors through $P$, say as $\tilde{\pi} : P \rightarrow I$.
{\em Macaulay 2} now shows that the kernel, say $L$,
of $A^3 \rightarrow I$ is generated by the columns
of the matrix:
$$
\left[
\begin{array}{rrrr}
c & -a & 0 &0\\
d & -b & -c &a\\
0 & 0 & -d & b
\end{array}
\right].
$$
(Actually, {\em Macaulay 2} gives 14 generators
but inspection shows that all are contained
in the submodule generated by these $4$).
Naturally $L$ contains the vector $ae_1 +(b+c)e_2+de_3$ and the quotient of $L$ by the cyclic submodule generated by this vector (which is a direct summand of $A^3$
and hence also of $L$) is the kernel of $\tilde{\pi}$. 
In other words, $ker(\tilde{\pi})$ is generated by $c\overline{e}_1 + d\overline{e}_2,
a\overline{e}_1+b\overline{e}_2 = -c\overline{e}_2-d\overline{e}_3$ and $a\overline{e}_2+b\overline{e}_3.$
We need to
explicitly identify this submodule of $P$ with $J$.

Now $J$ is generated by $c^2,a^2$ subject only to the relations given by the columns of
the matrix
$$
\left[
\begin{array}{rr}
a^2 & b^2\\
-c^2 & -d^2
\end{array}
\right].
$$
Thus there is a well-defined map $\tilde{\iota}: J \rightarrow P$ given by 
$\tilde{\iota}(c^2) = c\overline{e}_1+d\overline{e}_2,
\tilde{\iota}(a^2) = -a\overline{e}_2-b\overline{e}_3$. Using that $ac = a^2(pc+qd)+c^2(qa+rb),$ we see that $\tilde{\iota}(ac) = a\overline{e}_1+b\overline{e}_2$, and so
$\tilde{\iota}(J)=ker(\tilde{\pi})$.
Since we know that $ker(\tilde{\pi}) \cong J$ and a surjective endomorphism of a 
finitely generated module (over any commutative ring) is an isomorphism, $\tilde{\iota}$ is an isomorphism of $J$
onto $ker(\tilde{\pi})$.

Thus we now have an exact sequence
$$
\CD
0 @>>> J @>\tilde{\iota}>> P @>\tilde{\pi}>> I @>>> 0,
\endCD
$$
and we seek a section, say $\tilde{\sigma} : I \rightarrow P$, of $\tilde{\pi}.$
Since this seems difficult to split directly, we consider the above exact sequence tensored with $J$ which gives
$$
\CD
0 @>>> J^2 @>\hat{\iota}>> JP @>\hat{\pi}>> A @>>> 0,
\endCD
$$
which certainly is easier to split.
Here $\hat{\iota} = \tilde{\iota}|_{J^2}$ while $\hat{\pi} = a^{-2}\tilde{\pi}|_{JP}.$
Note that under $\hat{\pi}$, $a^2\overline{e}_1 \mapsto b^2$, $a^2\overline{e}_3 \mapsto a^2$, $c^2\overline{e}_1 \mapsto d^2$ and $c^2\overline{e}_3 \mapsto c^2$.
Thus a splitting of $\hat{\pi}$ is given
by the map $\hat{\sigma}: A \rightarrow JP$
given by $\hat{\sigma}(1) = xa^2\overline{e}_1 + wa^2\overline{e}_3
+zc^2\overline{e}_1 + yc^2\overline{e}_3,$
where $w,x,y,z$ satisfy $wa^2+xb^2+yc^2+zd^2=1$. Now tensor
with $I$ to get a splitting of our original
sequence. Thus, $\tilde{\sigma}$ is specified by
\begin{eqnarray*}
\tilde{\sigma}(a^2) &=& a^2(x\overline{e}_1 + w\overline{e}_3) + c^2(z\overline{e}_1 + y\overline{e}_3),{\text {\ \ \ and}}\\ \tilde{\sigma}(b^2) &=& b^2(x\overline{e}_1 + w\overline{e}_3) + d^2(z\overline{e}_1 + y\overline{e}_3).
\end{eqnarray*}
Finally conclude that the map $\tilde{\iota} + \tilde{\sigma}: J \oplus I \rightarrow P$ is an isomorphism.

Composing the isomorphisms of $A^2$ to $J \oplus I$ and $J \oplus I$ to $P$ we get
an explicit isomorphism $A^2 \rightarrow P$
given by:
\begin{eqnarray*}
e_1 &\mapsto& (a^2x+c^2z-cv)\overline{e}_1 + (at-dv)\overline{e}_2 + (a^2w+c^2y+bt)\overline{e}_3\\
e_2 &\mapsto& (b^2x+d^2z+cu)\overline{e}_1 + (du-as)\overline{e}_2 + (b^2w+d^2y-bs)\overline{e}_3
\end{eqnarray*}

It follows that an explicit completion of
the unimodular sequence $a,b+c,d$ is given by the matrix:
$$
\left[
\begin{array}{ccc}
a & a^2x+c^2z-cv & b^2x+d^2z+cu \\
b+c& at-dv & du-as \\
d& a^2w+c^2y+bt & b^2w+d^2y-bs
\end{array}
\right].
$$
And indeed, when we calculate the determinant of the above matrix in the ring $A$, it is $1$. 

Noticing that the elements $p,q,r$ of
$A$ do not occur in the above matrix we
hence try to compute the determinant in
the ring ${\mathbb Z}[a,b,c,d,s,t,u,v,w,x,y,z]/(ad-bc).$ The answer turns out to be
$(a^2s+b^2t+c^2u+d^2v)(a^2w+b^2x+c^2y+d^2z)$. Since the subring ${\mathbb Z}[a,b,c,d]$ of this ring
is isomorphic to ${\mathbb Z}[gj,gk,hj,hk]$, it follows that the determinantal identity of Proposition \ref{det} holds.



\section*{Acknowledgments}
I am extremely grateful to Prof. T. Y. Lam for
sparking my interest in problems of this
kind and to Prof. Pavaman Murthy for
encouraging me to write this up.


\begin{thebibliography}{KmrMrt2010}

\bibitem[Kmr1997]{Kmr1997} {\em A note on unimodular rows}, J. Algebra {\bf 191} (1997), no. 1,
228-234.

\bibitem[KmrMrt2010]{KmrMrt2010} N. Mohan Kumar and M. Pavaman Murthy, {\em Remarks on unimodular rows}, 
In `Quadratic forms, linear algebraic groups and cohomology', Eds. Jean-Louis Colloit-Thelene, Skip Garibaldi, R.Sujatha, Venapally Suresh,
Developments in Mathematics {\bf 18}, Springer (2010), 287-293. 

\bibitem[SwnTwb1975]{SwnTwb1975} Richard G. Swan and Jacob Towber, {\em A class of projective modules which are
nearly free}, J. Algebra {\bf 36} (1975), No. 3, 427-434.

\bibitem[Ssl1977]{Ssl1977} A. A. Suslin, 
{\em Stably free modules}, Mat. Sbornik. {\bf 102} (1977), no. 4, 537-550.

%
%
%


%

%




%







\end{thebibliography}
\end{document}